\documentclass[12pt,reqno]{amsart}
\usepackage{graphicx}
\usepackage{amsmath}
\usepackage{amssymb}
\usepackage{amscd}
\usepackage{mathrsfs}
\usepackage[all,cmtip]{xy}
 \usepackage{color}
\usepackage{amsthm}
\usepackage{caption}

\theoremstyle{definition}

\theoremstyle{definition}

\theoremstyle{definition}

\theoremstyle{plain}
\newtheorem{theorem}{Theorem}
\newtheorem{proposition}[theorem]{Proposition}
\newtheorem{corollary}{Corollary\hskip-0.5pt}
 
\theoremstyle{plain}

\newcommand\com[1]{}

\newcommand\Cc{{\let\mathcal\mathscr\mathcal C}}

\newcommand\D{{\mathcal D}}

\newcommand\E{\mathcal{E}}

\newcommand\g{{\frak g}}

\newcommand\La{\Lambda}

\newcommand\N{{\mathbb N}}

\newcommand\op[1]{\mathop{\rm #1}\nolimits}
\newcommand\ot{\otimes}

\newcommand\R{{\mathbb R}}

\newcommand\vp{\varphi}

\newcommand\Z{{\mathbb Z}}

\newcommand{\EW}{Einstein-Weyl }
\newcommand{\MS}{\mathcal{E}}

\newcommand{\sym}{\mathfrak{g}}

\def\presuper#1#2%
{\mathop{}%
	\mathopen{\vphantom{#2}}^{#1}%
	\kern-\scriptspace%
	#2}

\begin{document}

 \title[Differential invariants of 3D Einstein-Weyl structures]{Differential invariants of \\ Einstein-Weyl structures in 3D}
 \author{Boris Kruglikov, Eivind Schneider}
 \date{}
\address{\hspace{-17pt} Department of Mathematics and Statistics, UiT the Arctic University of Norway, Troms\o\ 90-37, Norway.\newline
E-mails: {\tt boris.kruglikov@uit.no, eivind.schneider@uit.no}. }

 \vspace{-14.5pt}
 \begin{abstract}
Einstein-Weyl structures on a three-dimensional manifold $M$ is given by a system $\mathcal E$ of PDEs on sections of a bundle over $M$. This system is invariant under the Lie pseudogroup $\mathcal G$ of local diffeomorphisms on $M$.
Two Einstein-Weyl structures are locally equivalent if there exists a local diffeomorphism taking one to the other.
Our goal is to describe the quotient equation $\mathcal E/\mathcal G$ whose solutions correspond to nonequivalent
Einstein-Weyl structures. The approach uses symmetries of the Manakov-Santini integrable system
and the action of the corresponding Lie pseudogroup.
 \end{abstract}

 \maketitle

%%%%%%%%%%%%%%%%%%%%%%%%%%%%%%%%%%%%%%%%%%%%%%%%%%%%%%%%%%%%%%%%%%%%%%%%%%%%
%0%

 % \tableofcontents
 \section*{Introduction}

A Weyl structure is a pair consisting of a conformal metric $[g]$ on a manifold $M$ and a symmetric linear connection
$\nabla$ preserving the conformal structure. This means
 \begin{equation}\label{E1}
\nabla g = \omega\otimes g
 \end{equation}
for some one-form $\omega$ on $M$ \cite{W}.
The Einstein-Weyl condition says that the symmetrized Ricci tensor of $\nabla$ belongs to the given conformal class:
 \begin{equation}\label{E2}
Ric^{\text{sym}}_\nabla=\Lambda g
 \end{equation}
for some function $\Lambda$ on $M$.
We call the pair $([g],\nabla)$ an \EW structure if it satisfies this \EW equation.

In this paper we restrict to three-dimensional manifolds. This is the first non-trivial case,
which is simultaneously the most interesting due to its relation with dispersionless
integrable systems \cite{Calderbank,FK}.
In addition, in dimension 3 the Einstein equation is trivial, meaning that all Einstein manifolds are space forms,
while the \EW equation is quite rich. The Einstein-Weyl equation has attracted a lot of attention due to its relations with
twistor theory, Lax integrability of PDE and mathematical relativity \cite{H,JT,DMT}.
It is worth mentioning that according to \cite{Cartan} the solution spaces of a third-order
scalar ODE with vanishing W\"unschmann and Cartan invariants carry a natural \EW structure.
We aim to solve the local equivalence problem for \EW structures in 3D.

The Einstein-Weyl equation is invariant under the Lie pseudogroup of local diffeomorphisms of $M$.
To construct the quotient of the action of this pseudogroup on the space of \EW structures we compute
the algebra of differential invariants, thus following the approach to the equivalence problem
as presented in \cite{T,ALV,O}.

We begin with general coordinate-free considerations in Section~\ref{S1}; this concerns conformal structures
of any signature. % and we outline the construction and count of differential invariants.
Then in Section \ref{S2} we specialize to the normal form of the pair $(g,\omega)$ introduced in \cite{DFK},
which expresses \EW structures locally by solutions of the modified Manakov-Santini system \cite{MS};
this is specific for the Lorentzian signature.
It will be demonstrated in Section \ref{S2} that the symmetry algebra of this PDE system coincides with the
algebra of shape preserving transformations for the metric in normal form \eqref{shape}.
Consequently, the problem is equivalent to computing differential invariants of the modified Manakov-Santini system
with respect to its symmetry pseudogroup.

In both cases we compute generators of the algebra of scalar rational differential invariants and derive the Poincar\'e function
counting the local moduli of the problem.
Section \ref{S1} and Sections \ref{S2}-\ref{S3}, supporting two different approaches to the same problem, can be read independently,
and the reader interested in geometry of the Manakov-Santini system can proceed straightforwardly to the latter sections.
Section \ref{S4} provides some particular solutions of the Manakov-Santini system, yielding several families of
non-equivalent \EW spaces parametrized by one or two functions of one argument. We stress that these explicit \EW structures
are non-homogeneous and not obtained by any symmetry reduction.
Appendix \ref{algebraicity} is devoted to the proof of a general theorem on algebraicity of the symmetry pseudogroup.

%%%%%%%%%%%%%%%%%%%%%%%%%%%%%%%%%%%%%%%%%%%%%%%%%%%%%%%%%%%%%%%%%%%%%%%%%%%%
%1%
 \section{Differential invariants of \EW structures}\label{S1}

In this section we discuss the general coordinate-free approach to computation of differential invariants of \EW structures in 3D. The conformal structure can be both of Riemanian and Lorentzian signature.
We refer to \cite{O,KLV,KL1} for the basics of jet-theory, Lie pseudogroups and differential invariants.

%%%%%%%%%%%%%%%%%%%%%%
\subsection{Setup of the problem}\label{S1.1}

Let a Lie pseudogroup $G$ act on the space of jets $\mathcal{J}$ or a differential equation considered as a co-filtered submanifold in it (also know as diffiety);
we keep the same notation for the latter in this setup.

A differential invariant of order $k$ is a smooth function $I$ on $\mathcal{J}_k$ that is constant on orbits of the $G$-action.
If the pseudogroup $G$ is topologically connected (the same as path-connected), then the definition of differential invariant
is equivalent to the constraint $L_{X^{(k)}}I=0$ for every $X$ in the Lie algebra sheaf $\g$ corresponding to $G$,
where $X^{(k)}$ denotes the prolongation of the vector field $X$ to $k$-jets.

It turns out that in our problem, the pseudogroup $G$, the space $\mathcal{J}$ and the action are algebraic in
the sense of \cite{KL2} (for the data in this section this follows from the definition, and for the objects
in the following sections it follows from a general theorem in the appendix).
Moreover, the action of $G$ is transitive on the base and $\mathcal{J}$ is irreducible. Under these conditions,
the global Lie-Tresse theorem \cite{KL2} implies that the space of rational differential invariants is finitely
generated as a differential field, i.e.\ there exist a finite number of differential invariants and
invariant derivations that algebraically generate all other invariants.
In addition, the theorem states that differential invariants separate orbits in general position, thus solving
the local equivalence problem for generic structures.

In our work the pseudogroup $G$ (and later $\mathcal{G}$) will be connected in the Zariski topology.
 % (it will be even irreducible)
In this case the condition that a rational function $I$ is a differential invariant
is equivalent to the constraint $L_{X^{(k)}}I=0$ for every $X$ from the Lie algebra sheaf $\g$ of $G$.

Weyl structures are given by triples $(g,\nabla,\omega)$ satisfying relation \eqref{E1}. Let us note that essentially
two of the structures are enough to recover the third one. Indeed, $g$ and $\nabla$ give $\omega$ by \eqref{E1}.
Also, $g$ and $\omega$ give $\nabla=\nabla^g+\rho(\omega)$, where
$2\rho(\omega)(X,Y)=\omega(X)Y+\omega(Y)X-g(X,Y)\omega^\sharp_g$. In coordinates this relates the Christoffel symbols
of $\nabla$ and the Levi-Civita connection $\nabla^g$:
 $$
\Gamma_{ij}^k=\gamma_{ij}^k+\tfrac12(\omega_i\delta_j^k+\omega_j\delta_i^k-g_{ij}\omega^k).
 $$
Finally, the same formula expresses $\nabla^g$ from $\nabla$ and $\omega$.
It is known that % affine transformations generically coincide with isometries and so $\nabla^g$ determines $g$
if $(M,g)$ is holonomy irreducible and admits no parallel null distribution, then $\nabla^g$ determines $g$
up to homothety. This recovers $[g]$ in this generic case.

It is not true though that $k$-jet of one pair correspond to $k$-jet of another representative pair,
the jets are staggered in this correspondence. In what follows we will restrict to equivalence classes of
pairs $(g,\omega)$: when the representative of $[g]$ is changed $g\mapsto f^2 g$,
the one-form also changes $\omega\mapsto\omega+2df/f$.

Thus the space of moduli of Weyl structures can be considered as the space $\mathcal{W}$ of pairs $(g,\omega)$
modulo the pseuodogroup $G=\op{Diff}_{\text{loc}}(M)\times C^\infty_{\neq0}(M)$ consisting of pairs $(\varphi,f)$
of a local diffeomorphism $\varphi$ and a nonzero function $f$. The action is clearly algebraic.

%%%%%%%%%%%%%%%%%%%%%%
\subsection{Weyl structures}\label{S1.2}

The $G$-action has order 1, i.e.\ for any point $a\in M$ the stabilizer subgroup in $(k+1)$-jets $G_a^k$ acts
on the space $\mathcal{W}_a^k$ of $k$-jets of the structures at $a$.
For a point $a_k\in\mathcal{W}^k_a$ denote $\op{St}^{k+1}_{a_k}$ its stabilizer in $G^{k+1}_a$.
Let also $g_k=\op{Ker}(d\pi_{k,k-1}:T_{a_k}\mathcal{W}^k_a\to T_{a_{k-1}}\mathcal{W}^{k-1}_a)$ denote the symbol
of the space of Weyl structures. The differential group $G$ has the following co-filtration:
 $$
0\to\Delta_k\longrightarrow G^k_a \longrightarrow G^{k-1}_a\to 1,
 $$
where $\Delta_k=S^kT^*_a\otimes T_a\oplus S^kT^*_a$ for $k>1$, and we abbreviate $T_a=T_aM$.
For $k=1$, $G^1_a=\Delta_1=\op{GL}(T_a)\oplus T^*_a\oplus\R^\times$.

The 0-jet $a_0$ is the evaluation $(g_a,\omega_a)$. By $G^1_a$-action the second component can
be made zero, and the first component rescaled. The action of $GL(T_a)$ on the conformal class $[g_a]$ yields $\op{St}^1_{a_0}=CO(g_a)$.

The group $\Delta_2=S^2T^*_a\ot T_a\oplus S^2T^*_a$ acts on the symbol $g_1=T_a^*\ot S^2T^*_a\oplus T_a^*\ot T_a^*$ of $\mathcal{W}$.
This action is free and $g_1/\Delta_2=\La^2T^*_a$. This is the space where $Ric_\nabla^{\text{skew}}=\frac32d\omega$ \cite{JT} lives.
The stabilizer from the previous jet-level $CO(g_a)$ acts with an open orbit, i.e.\ there are no scalar invariants.
There are however the following vector and tensor invariants: $L^1=\op{Ker}(d\omega)$, $\Pi^2=L^1_\perp$ (generically $L^1$ is non-null and so transveral to $\Pi^2$)
and a complex structure $J=g^{-1}d\omega$ on $\Pi$, where the representative $g$ is normalized so that $\|d\omega\|^2_g=1$.
The stabilizer $\op{St}^2_{a_1}$ is either $SO(2)\times\Z_2$ or $SO(1,1)\times\Z_2$.

Starting from $k\ge2$ the action of $G^{k+1}_a$ on a Zariski open subset of $\mathcal{W}^k_a$ is free, i.e.\ the stabilizer is resolved:
$\op{St}^{k+1}_{a_k}=0$ for generic $a_k\in\mathcal{W}^k_a$.
This can be seen by the exact seqences approach as in \cite{LY}, and can be verified directly. The metric $g$ chosen with the above normalization
is the unique conformal representative, then $\omega$ is defined uniquely as well, and we can have the following canonical frame on $M$,
defined by a Zariski generic $a_2$: $e_1\in L^1$  normalized by $\omega(e_1)=1$, $e_2=\pi(\omega^\sharp_g)$ with $\pi:T_a\to\Pi^2$ being the orthogonal
projection along $L^1$, and $e_3=Je_2$. Coefficients of the structure $(g,\omega)$ written in this frame give a complete set of scalar rational differential invariants.

The count of them is as follows. Let $s_k$ be the number of independent differential invariants of order $\leq k$, which coincides with the transcendence degree of the
field of order $\leq k$ rational differential invariants.
Let $h_k=s_k-s_{k-1}$ be the number of ``pure'' order $k$ invariants. Then $h_0=h_1=0$ and $h_2=\dim g_2-\dim\Delta_3-\dim SO(2)=54-40-1=13$ and
$h_k=\dim g_k-\dim\Delta_{k+1}=9\binom{k+2}{2}-4\binom{k+3}2=\frac12(5k^2+7k-6)$ for $k>2$.
These numbers are encoded by the Poincar\'e function
 $$
P(z)=\sum_{k=0}^\infty h_kz^k=\frac{(13-9z+z^3)z^2}{(1-z)^3}.
 $$

%%%%%%%%%%%%%%%%%%%%%%
\subsection{Einstein-Weyl structures}\label{S1.3}

The Einstein-Weyl equation \eqref{E2} is a set of 5 equations on 8 unknowns, which looks like an underdetermined system. However its $\op{Diff}_\text{loc}(M)$-invariance reduces the number of unknowns to 8-3=5 and makes it
a determined system -- formally this follows from the normalization of \cite{DFK}.

Denote this equation by $\mathcal{EW}$. The number of its determining equations of order $k$ is $5\binom{k}{2}$.
Let $\tilde{g}_k=\op{Ker}(d\pi_{k,k-1}:T_{a_k}\mathcal{EW}^k_a\to T_{a_{k-1}}\mathcal{EW}^{k-1}_a)$ be the symbol of the system.
Its dimension is $\dim\tilde{g}_k=\dim g_k-5\binom{k}{2}$.

The action of $G^{k+1}_a$ on $\mathcal{EW}^k_a$ is still free starting from $k\ge2$ and this implies
that the number of ``pure order'' $k$ invariants is: $\bar{h}_0=\bar{h}_1=0$, $\bar{h}_2=13-5=8$,
and $\bar{h}_k=h_k-5\binom{k}{2}=3(2k-1)$ for $k>2$. The corresponding Poincar\'e function is equal to
 $$
\bar{P}(z)=\sum_{k=0}^\infty\bar{h}_kz^k=\frac{(8-z-z^2)z^2}{(1-z)^2}.
 $$
We again have the canonical frame $(e_1,e_2,e_3)$, and this yields all scalar rational differential invariants of $\mathcal{EW}$.

%%%%%%%%%%%%%%%%%%%%%%%%%%%%%%%%%%%%%%%%%%%%%%%%%%%%%%%%%%%%%%%%%%%%%%%%%%%%
%2%
 \section{Einstein-Weyl structures via an integrable system}\label{S2}

In this section we study the Lie algebra $\sym$ of point symmetries of the modified Manakov-Santini system $\MS$, defined by (\ref{MS}), which describes % locally all Lorentzian signature
three-dimensional \EW structures of Lorentzian signature. We calculate the dimensions of generic orbits of $\sym$. The Einstein-Weyl structures corresponding to solutions of $\MS$ are of special shape (\ref{shape}), and we compute the Lie algebra $\mathfrak h$ of vector fields preserving this shape (ansatz).
It turns out that the lift of $\mathfrak h$ to the total space $E$ is exactly $\sym$, whence $\mathfrak h\simeq\sym$.

%%%%%%%%%%%%%%%%%%%%%%
\subsection{A modified Manakov-Santini equation and its symmetry}\label{S2.1}

By \cite{DFK} any Lorentzian signature Einstein-Weyl structure is locally of the form
 \begin{equation}
\begin{aligned}
g &= 4 dt dx+2 u dt dy-(u^2+4v) dt^2 - dy^2 \\
\omega &= (u u_x+2 u_y+4v_x) dt - u_x dy
\end{aligned}\label{shape}
 \end{equation}
where $u$ and $v$ are functions of $(t,x,y)$ satisfying
 \begin{equation}
\begin{split}
F_1&=(u_t+u u_y+v u_x)_x- (u_y)_y=0, \\ F_2&=(v_t + v v_x-u v_y)_x - (v_y-2uv_x)_y=0.
% F_1&=(u_t+u u_y+v u_x)_x- (u_y)_y=0, \\ F_2&=(v_t+u v_y+v v_x)_x - (v_y)_y+2(u_yv_x-u_xv_y)=0.
\end{split} \label{MS}
 \end{equation}

This system, derived in the proof of Theorem 1 in \cite{DFK}, is related to the
Manakov-Santini system \cite{MS} by the change of variables $(u,v) \mapsto(v_x,u-v_y)$ and potentiation.
We will refer to it as the modified Manakov-Santini system.

Note that normalization of the coefficient of $dy^2$ in $g$ to be $-1$ gives a representative of the
conformal class $[g]$, reducing the $C^\infty_{\neq0}(M)$-component of the pseudogroup $G$ from the previous section.

Let $M=\R^3(t,x,y)$. We treat the pair $(g,\omega)$ as a section of the bundle
 \[\pi \colon E=M \times \mathbb R^2(u,v) \to  M .\]
This is a subbundle of $S^2 T^*M \oplus T^*M$, considered in Section \ref{S1}.

Einstein-Weyl structures correspond to sections of $\pi$ satisfying (\ref{MS}). Consider the system (\ref{MS}) as a nonlinear subbundle $\MS_2=\{F_1=F_2=0\}$ of the
jet bundle $J^2\pi$, and denote its prolongation by $\MS_k \subset J^k\pi$. The notation $\MS_0=J^0\pi=E$,
$\MS_1=J^1\pi$ will be used. Let $\MS \subset J^\infty \pi$
denote the projective limit of $\MS_k$.

The dimension of $J^k\pi$ is $3+2 \binom{k+3}{3}$, while the number of equations determining $\MS_k$ is $2 \binom{k+1}{3}$. The system $\E$ is determined, so these equations are independent, whence
 $$
\dim\MS_k = \dim J^k\pi-2\tbinom{k+1}{3} = 3+2(k+1)^2,\ k\ge2.
 $$
For $k=0,1$ we have $\dim\MS_0=5$, $\dim\MS_1=11$.

A vector field $X$ on $E$ is an (infinitesimal point) symmetry of $\MS$ if its prolongation $X^{(2)}$ to $J^2\pi$ is tangent to $\MS_2$, in other words if it satisfies the Lie equation
 $$
(L_{X^{(2)}}F_i)|_{\E_2}=0 \text{ for }i=1,2.
 $$
Decomposing this by the fiber coordinates of $\E_2\to E$, we get an overdetermined system of linear PDEs
on the coefficients of $X$. This system can be explicitly solved, and the result is as follows.
 % The space of such vector fields constitute a Lie algebra, and the vector fields are listed in the following theorem.

 \begin{theorem}\label{th:symmetries}
The Lie algebra $\sym$ of symmetries of $\MS$ has the following generators,
involving five arbitrary functions $a=a(t),\dots,e=e(t)$:
 \begin{align*}
X_1(a)&= a \partial_x+\dot a \partial_v\\
X_2(b)&= b \partial_y+\dot b \partial_u \\
X_3(c)&= y c \partial_x-2c \partial_u+(u c+y \dot c) \partial_v\\
X_4(d)&= d\partial_t +\frac12\dot d y \partial_y+\frac12(y\ddot d -u\dot d) \partial_u -\dot d v\partial_v\\
X_5(e)&= (y^2\dot e+2xe)\partial_x+ye\partial_y+(ue-3y\dot e)\partial_u+(y^2\ddot e+2yu\dot e+2ve+2x\dot e)\partial_v
 \end{align*}
Table \ref{CommutationRelations} shows the commutation relations of $\g$.
 \end{theorem}
It follows from the table that $\sym$ is a perfect Lie algebra: $[\sym,\sym]=\sym$. We also see that the splitting $\sym=\sym_0 \oplus \sym_1 \oplus \sym_2$, with $\sym_0=\langle X_4,X_5\rangle, \sym_1=\langle X_2,X_3\rangle, \sym_2=\langle X_1 \rangle$, gives a grading of $\sym$, i.e. $[\sym_i, \sym_j] \subset \sym_{i+j}$ ($\sym_i=0$ for $i \notin \{0,1,2\}$).

\bgroup
\def\arraystretch{1.3}
\begin{table}
	\begin{tabular}{|c||c|c|c|c|c|}
		\hline
		& $X_1(g)$ & $X_2(g)$ & $X_3(g)$ & $X_4(g)$ & $X_5(g)$  \\ \hline \hline
		$X_1(f)$	& $0$ & $0$  & $0$ & $X_1(-g \dot f )$ & $X_1(2 f g)$ \\ \hline
		$X_2(f)$	& $*$ & $0$ & $X_1(fg)$ & $X_2( \frac{f \dot g}{2}-g \dot f)$ & $X_2(fg)+X_3(2 f \dot g)$ \\ \hline
		$X_3(f)$	& $*$ & $*$ & $0$ & $X_3( -g \dot f-\frac{f \dot g}{2})$ & $X_3(fg)$ \\ \hline
		$X_4(f)$	& $*$ & $*$ & $*$ & $X_4(f \dot g-g \dot f)$ & $X_5(f \dot g)$\\\hline
		$X_5(f)$    & $*$ & $*$ & $*$ & $*$ & $0$ \\\hline
	\end{tabular} 	
	\caption{The structure of the symmetry Lie algebra 	$\sym$.}
	\label{CommutationRelations}
\end{table}
\egroup

Consider the action of the Lie pseudogroup $\mathcal G_\text{top}$ on $E$ defined by
\begin{align*}
t &\mapsto D(t) ,\\
x &\mapsto E(t)^2 x+E(t) E'(t) y^2+C(t) y+A(t),\\
y &\mapsto \sqrt{D'(t)} E(t) y+B(t),\\
u &\mapsto \frac{E(t)}{\sqrt{D'(t)}} u -\frac{y}{E(t)^2}\frac{d}{dt}\frac{E(t)^3}{\sqrt{D'(t)}}+\frac{B'(t)}{D'(t)}-\frac{2C(t)}{E(t) \sqrt{D'(t)}} ,\\
v &\mapsto \frac{E(t)^2}{D'(t)} v + \frac{C(t)+2E(t) E'(t) y}{D'(t)} u + \frac{E(t) E''(t)-3 E'(t)^2}{D'(t)} y^2\\&+ \frac{E(t)^4}{D'(t)}\frac{d}{dt}\frac{C(t)}{E(t)^4}y+ \frac{2E(t) E'(t)}{D'(t)}x+\frac{E(t)^2 A'(t)-C(t)^2}{D'(t) E(t)^2},
\end{align*}
where $D \in \text{Diff}_{\text{loc}}^+(\mathbb R)$ is an orientation-preserving local diffeomorphism of $\mathbb R$ and $A,B,C,E$ are smooth functions with the same domain as $D$ and $E(t) > 0$ for every $t$ in its domain.

This Lie pseudogroup is topologically connected and has $\sym$ as its Lie algebra of vector fields. However $\mathcal G_\text{top}$ is not algebraic. Since the global Lie-Tresse theorem holds for algebraic Lie pseudogroups, we consider the Zariski closure of $\mathcal G_\text{top}$, denoted by $\mathcal G_Z$. The subgroup $\mathcal G_\text{top}$ is normal in $\mathcal G_Z$ and  $\mathcal G_Z / \mathcal G_\text{top} = \mathbb Z_2 \times \mathbb Z_2$ is generated by reflections $(t,x,y) \mapsto (-t,-x,-y)$ and $(y,u)\mapsto (-y,-u)$. Thus it can be argued, also from a geometric viewpoint, that it is more natural to consider $\mathcal G_Z$ instead of $\mathcal G_\text{top}$.
Since $\mathcal G_Z$ is the Lie pseudogroup we will be interested in, we simplify the notation to $\mathcal G$.

%%%%%%%%%%%%%%%%%%%%%%
\subsection{Dimension of generic orbits}\label{S2.3}

Denote by $\mathcal O_k$ a generic orbit of the $\mathcal G$-action on $\MS_k$.
Its topologically-connected component is an orbit of the prolongation $\sym^{(k)}$ of $\sym$, and so we consider the action of the latter.

% Let $\sym^{(k)}$ be the prolongation of $\sym$. It consists of vector fields on $J^k\pi$, tangent to $\E_k$.
% As in Section \ref{S1.1}, rational differential invariants of $\sym^{(k)}$ coincide with those of $\mathcal G$.
% Denote by $\mathcal O_k$ a generic orbit of the $\mathcal G$-action on $\MS_k$.
% Its connected component is an orbit of $\sym^{(k)}$.

The Lie algebra $\sym$ acts transitively on $J^0\pi$ and $\sym^{(1)}$ acts locally transitively
on $J^1\pi$ (the hyperplane given by $u_x=0$ is invariant). A generic orbit of $\sym^{(2)}$ on both $\MS_2$ and $J^2\pi$ has dimension $18$.
The next theorem describes the orbit dimensions for every $k$.%, and its proof partially follows that of \cite[Section 6.1]{KS}.

 \begin{proposition}
A generic orbit $\mathcal O_k$ of the $\g^{(k)}$-action on $\MS_k$ satisfies:
 \[
\dim \mathcal O_0=5,\qquad \dim \mathcal O_1=11,\qquad  \dim \mathcal O_k = 5k+8 , \quad k \geq 2.
 \]
 \end{proposition}

 \begin{proof}
Consider the point  $(t,x,y,u,v)=(0,0,0,0,0)\in E$, and denote its fiber under the projection $\MS_k \to E$ by $S_k$.
Since $\sym$ acts transitively on $E$, every orbit of $\sym^{(k)}$ in $\MS_k$ intersects $S_k$ at some
point $\theta_k \in S_k$. Denote by $\mathcal{O}_{\theta_k}$ the $\g^{(k)}$-orbit through $\theta_k \in S_k$.
We have $T_{\theta_k}\mathcal{O}_{\theta_k}=\op{span}\{X_i^{(k)}(f_i)_{\theta_k}:f_i\in C^\infty(\R),i=1,...,5\}$. Here and below $X_i^{(k)}(f)_{\theta_k}$ denotes the prolongation of the vector field $X_i(f)$ to $J^k \pi$, evaluated at the point $\theta_k$.
%To find the orbit dimension, we can compute the dimension of the subspace in $T_{\theta_k} \MS_k$ spanned by the vector fields in $\sym^{(k)}$.
	
The $k$-th prolongation of a vector field $X$ has the coordinate form
 \begin{equation}\label{eq:prolong}
X^{(k)} = \sum_{i=1}^3 \alpha^i \D_i^{(k+1)}+\sum_{|\sigma|\leq k} \left( \D_\sigma (\phi_u) \partial_{u_\sigma}+\D_\sigma (\phi_v) \partial_{v_\sigma}\right).
 \end{equation}
Here $\sigma$ is a multi-index, $\D_\sigma$ is the iterated total derivative, $\D_i^{(k+1)}$ is the truncated total derivative as a derivation on $k$-jets\footnote{The truncated total derivative is given by $\D_i^{(k+1)} = \partial_{x^i} + \sum_{|\sigma|\leq k} (u_{\sigma i} \partial_{u_\sigma} + v_{\sigma i} \partial_{v_\sigma})$.}, $\alpha^i = dx^i(X)$ with the notation $(x^1,x^2,x^3)=(t,x,y), u_\sigma=u_{x^\sigma},v_\sigma=v_{x^\sigma}$, and the functions $\phi_u=\omega_u(X)$, $\phi_v=\omega_v(X)$ are components of the generating section $\phi=(\phi_u,\phi_v)$ for $X$, where
 \begin{align*}
\omega_u=du-u_t dt-u_x dx-u_y dy, \qquad \omega_v=dv-v_t dt-v_x dx-v_y dy .
 \end{align*}
Below we denote by $Y_i^k(m) =\frac{1}{m!} X_i^{(k)}(t^m)$ for $i=1,...,5$, the vector fields on $\E_k$.
	
Consider first the vector field $X_1(a)$. Its generating section is
 \[\phi_1=(-u_x a(t),\; \dot a(t) - v_x a(t)).\]
This and (\ref{eq:prolong}) implies that the vector $X_1^{(k)}(a)_{\theta_k}$ depends only on $a(0),...,a^{(k+1)}(0)$.
Therefore $\op{span}\{X_1^{(k)}(a)_{\theta_k}:a\in C^\infty(\R)\}=\op{span}\{Y_1^k(m)_{\theta_k}:m=0,\dots,k+1\}$.
	
Repeating this argument for $X_2(b),...,X_5(e)$ we conclude that the subspace
$T_{\theta_k}\mathcal{O}_{\theta_k}\subset T_{\theta_k} \MS_k$ is spanned by
 \begin{equation}
V_k=\{Y_1^k(m),Y_2^k(m),Y_3^k(n),Y_4^k(m),Y_5^k(n) : m\leq k+1, n\leq k\}
 \end{equation}
evaluated at $\theta_k$. This gives the upper bound $5k+8 = |V_k|$ for $\dim \mathcal{O}_k$. (For $k=0,1$ the orbit dimension is bounded even more by $\dim \MS_0=5$ and $\dim \MS_1=11$.)
	
We use induction to show that there exist orbits of dimension $5k+8$ for $k \geq 2$. Due to lower semicontinuity of matrix rank, an orbit in general position will then also have the same dimension.
We choose $\theta_k$ to be given by $u_x=1,u_{xx}=1$ and all other jet-variables set to $0$.
For the induction step assume that all vectors in the set $V_k$ are independent, and hence
$\dim\mathcal{O}_{\theta_k}=5k+8$. For $k=2$ this is easily verified in Maple.
The five vectors
 \begin{align*}
Y_1^{k+1}(k+2)_{\theta_{k+1}} = \partial_{v_{t^{k+1}}}, &\qquad
Y_2^{k+1}(k+2)_{\theta_{k+1}} = \partial_{u_{t^{k+1}}}, \\
Y_3^{k+1}(k+1)_{\theta_{k+1}} = \partial_{v_{t^k y}}-2 \partial_{u_{t^{k+1}}}, &\qquad
Y_4^{k+1}(k+2)_{\theta_{k+1}} = \frac12 \partial_{u_{t^k y}},\\
Y_5^{k+1}(k+1)_{\theta_{k+1}} = -3 \partial_{u_{t^{k} y}} &+2 \partial_{v_{t^{k} x}} +2 \partial_{v_{t^{k-1} y^2}}
 \end{align*}
are independent and tangent to the fiber of $S_{k+1}$ over $\theta_k \in S_k$.
Therefore they are independent with the prolonged vector fields from $V_k$ at $\theta_{k+1}$.
Thus $\dim \mathcal O_{\theta_{k+1}}=5k+8+5=5(k+1)+8$, completing the induction step and the proof.
\end{proof}

%%%%%%%%%%%%%%%%%%%%%%
\subsection{Shape-preserving transformations}\label{S2.4}

The ansatz (\ref{shape}) for \EW structures on $M$ is not invariant under arbitrary local diffeomorphisms of $M$,
and we want to determine the pseudogroup preserving this shape of $(g,\omega)$. Its Lie algebra
sheaf is given as follows.

%Let $g$ and $\omega$ be as in (\ref{shape}). The vector field $X$ preserves the shape (\ref{shape}) of $g$ and $\omega$ if there exists a function $\lambda \in C^\infty(M)$ with the properties that $L_X g=\lambda g$ and $L_X \omega = d \lambda$.

 \begin{theorem}
The Lie algebra $\mathfrak h$ of vector fields preserving shape (\ref{shape}) of $(g, \omega)$ has the following generators,
involving five arbitrary functions $a=a(t),...,e=e(t)$:
 \begin{align}\label{eq:shapepreserving}
a \partial_x,\quad
b \partial_y, \quad
y c \partial_x ,\quad
d  \partial_t +\frac{1}{2} \dot d  y \partial_y, \quad
(y^2 \dot e+2 x e) \partial_x+y e \partial_y.
 \end{align}
 \end{theorem}

 \begin{proof}
Let $X=\alpha(t,x,y) \partial_t+\beta(t,x,y) \partial_x+\gamma(t,x,y) \partial_y$ be a vector field on $M$ preserving
the shape of $(g, \omega)$, and $\varphi_\tau$ its flow. The pullback of $g$ through $\varphi_{\tau}$ has the same shape, up to a conformal factor $f^\tau$, so that
 \[ \varphi_{\tau}^* g = f^\tau (4 dt dx+2 u^\tau dt dy- ((u^\tau)^2+4v^\tau) dt^2-dy^2),\]
where $f^\tau, u^\tau, v^\tau$ are $\tau$-parameteric functions of $t,x,y$ with $f^0=1, u^0=u, v^0=v$.
Denote $\chi=\frac{d}{d\tau}\big|_{\tau=0} f^\tau, \mu=\frac{d}{d\tau}\big|_{\tau=0} u^\tau, \nu=\frac{d}{d\tau}\big|_{\tau=0} v^\tau$. Then the Lie derivative is
 \[ L_X g= \chi g + 2\mu dt dy - (2u\mu+4\nu) dt^2. \]
Similarly, from $\varphi_\tau^*\omega = \omega^\tau+d \log f^\tau$, we obtain the formula
 \[ L_X \omega = (u_x\mu+u\mu_x +2\mu_y +4\nu_x) dt -\mu_x dy +d\chi. \]
	
These restrictions yield an overdetermined system of differential equations on $\alpha$, $\beta$ and $\gamma$ whose solutions give exactly the vector fields (\ref{eq:shapepreserving}).
 \end{proof}

 \com{
 The following Lie pseudogroup action on $M$ has $\mathfrak h$ as its Lie algebra of vector fields on $M$:
 \begin{align*}
 t &\mapsto D(t) \\
 x &\mapsto E(t)^2 x+E(t) E'(t) y^2+C(t) y+A(t)\\
 y &\mapsto |D'(t)|^{1/2} E(t) y+B(t)
 \end{align*}
 where $D \in \text{Diff}_{\text{loc}}(\mathbb R)$ and $A,B,C,E$ are smooth functions with the same domain as $D$ and $E(t) \neq 0$ for every $t$ in its domain.
 }

The Lie algebra $\mathfrak{h}$ of vector field on $M$ can be naturally lifted to
the Lie algebra $\hat{\mathfrak{h}}$ on the total space $E$. Let $X\in\mathfrak{h}$.
Its lift $\hat X=X+A\partial_u+B\partial_v\in\hat{\mathfrak{h}}$ is computed  as follows.
The pullback of $g$ to $E$ is a horizontal symmetric two-form $\hat{g}$.
Then the condition $L_{\hat X}\hat{g}= \chi \hat{g}$ uniquely determines the coefficients $A,B$.

Applying this to the general vector field
$X=2d \partial_t+(a+yc+2xe+y^2 \dot e) \partial_x+(b+y \dot d+ye) \partial_y\in\mathfrak{h}$
we get $\chi=2(e(t)+\dot d(t))$.
Moreover for the pull-back $\hat\omega$ of $\omega$ and the prolongation of the vector field ${\hat X}$
we get $L_{{\hat X}^{(1)}}\hat\omega = d \chi$.
Comparing the resulting $A$ and $B$ with the vector fields in Theorem \ref{th:symmetries}, we conclude:

 \begin{corollary}
The lift $\hat{\mathfrak{h}}$ of the Lie algebra $\mathfrak h$ of shape-preserving vector fields is exactly the Lie algebra $\sym$ of point symmetries of $\MS$.
 \end{corollary}

%This tells us that the system we are using for describing EW structures is good.
% It is not so surprising that the symmetry algebra of $\mathcal E$ contains all vector fields preserving the shape (\ref{shape}), since the symmetry algebra of the original Einstein-Weyl equation contains all vector fields on the manifold. The opposite inclusion is more surprising. There could be additional point symmetries as a result of changing the equation. As an example, the original Manakov-Santini system, which is very similar to $\MS$,  has point symmetries which are tangent to the fibers of $M \times \mathbb R^2 \to M$.

Let us reformulate our lift of the algebra $\mathfrak{h}$ using integrability of system \eqref{MS}.
Its Lax pair is given by a rank 2 distribution
$\tilde{\Pi}^2=\op{span}\{\partial_y-\lambda\partial_x+n\partial_\lambda,
\partial_t-(\lambda^2-u\lambda-v)\partial_x+m\partial_\lambda\}$
on $\mathbb{P}^1$-bundle $\tilde{M}$ over $M$, which is Frobenius-integrable in virtue of \eqref{MS}
(the form of $m,n$ is not essential here, see \cite{DFK}).
The fiber can be identified with the projectivized null-cone of $g$.
The coordinate $\lambda$ along it is called the spectral parameter.
The action of $\mathfrak{h}$ on $M$ induces the action on $\tilde{M}$
and hence on $\tilde{\Pi}^2$.
Since the plane $\tilde{\Pi}^2_{(t,x,y,\lambda)}$ is projected to the plane
$\Pi^2=\op{Ann}(dx+\lambda dy+(\lambda^2-u\lambda-v)dt)$, this in turn gives the action on $u,v$, i.e. the required lift.

%%%%%%%%%%%%%%%%%%%%%%%%%%%%%%%%%%%%%%%%%%%%%%%%%%%%%%%%%%%%%%%%%%%%%%%%%%%%
%3%
\section{Differential invariants of $\E$}\label{S3}

In this section we determine generators of the field of scalar rational differential invariants of the equation $\E$ with respect to its
symmetry pseudogroup $\mathcal{G}$. We also compute the Poincar\'e function of the $\mathcal{G}$-action, counting moduli of the problem,
and discuss solution of the equivalence problem for \EW structures written in form \eqref{shape}.

%%%%%%%%%%%%%%%%%%%%%%
\subsection{Hilbert polynomial and Poincar\'e function}\label{S3.1}

The number $s_k$ of independent differential invariants of order $k$ is equal to the codimension of a generic orbit $\mathcal O_k\subset\MS_k$.
Since, as in Section \ref{S1.1}, rational differential invariants of $\mathcal G$ coincide with those of $\sym^{(k)}$,
we can compute $s_k$ using the results from Section \ref{S2.3}:
 \[ s_k = \dim \MS_k - \dim \mathcal O_k  =2k^2-k-3, \qquad k \geq 2\]
Due to local transitivity $s_0=s_1=0$.

The difference $h_k=s_k-s_{k-1}$ counts the number of invariants of ``pure'' order $k$. It is given as folows:
$h_0=h_1=0$, $h_2=3$ and $h_k=4k-3$ for $k>2$. The Hilbert polynomial is the stable value of $h_k$: $H(k)=4k-3$.

These numbers can be compactified into the Poincar\'e function:
 $$
P(z)=\sum_{k=0}^\infty h_kz^k=\frac{(3+3z-2z^2)z^2}{(1-z)^2}.
 $$

%%%%%%%%%%%%%%%%%%%%%%
\subsection{Invariant derivations and differential invariants}\label{S3.2}

All objects we treat in this section will be written in terms of ambient coordinates on $J^k\pi\supset\E_k$.

From the previous section, we know that there exist three independent rational differential invariants of order two.
The second-order invariants are generated by
 \begin{align*}
I_1 = \frac {u_{xy}+v_{xx}}{u_x^{2}},\quad
I_2 &= \frac {u_x^{2}u_{xy}+u_x u_{xx}v_x+u_{xx}u_{yy}-u_{xy}^{2}}{u_x^{4}},\\
I_3 &= \frac {u_x^{2}v_{xx}-u_x u_{xx}v_x+u_{xx}v_{xy}-u_{xy}v_{xx}}{u_x^{4}}.
 \end{align*}
In order to generate all differential invariants, we also need invariant derivations. These are derivations on the algebra of differential invariants commuting with $\mathcal G$. It is easily checked that
\begin{align*}
\nabla_1 &= \frac{u_x}{u_{xx}} D_x, \qquad
\nabla_2 = \frac{1}{u_x} \left(\frac{u_{xy}}{u_{xx}} D_x-D_y   \right),\\
%\nabla_3 &= \frac{1}{u_x^3} \Big(u_{xx} D_t+\left(v u_{xx}+v_x u_x+u_{yy}\right) D_x+ \left( u u_{xx} + u_x^2 -2 u_{xy}  \right) D_y    \Big)\\
\nabla_3 &= \frac{1}{u_x^3} \Big(u_{xx} D_t+\left((v u_{x})_x+u_{yy}\right) D_x+ \left( u u_{x}-2 u_{y}    \right)_x D_y    \Big)
\end{align*}
%The old coefficient of D_x in \nabla_3 was 2vu_{xx}+(u_y+2v_x) u_x+u u_{xy}+u_{tx}
are three independent invariant derivations. Their commutation relations are given by
 \begin{align*}
[\nabla_1,\nabla_2] = -\nabla_2, \qquad
&[\nabla_1,\nabla_3]=-K_3 \nabla_1+ (K_1-2K_2) \nabla_2+ K_1 \nabla_3 , \\
&[\nabla_2,\nabla_3]= K_4 \nabla_1+K_3 \nabla_2+K_2 \nabla_3,
 \end{align*}
where
 \begin{align*}
K_1&= \frac{u_x u_{xxx}}{u_{xx}^2}-3 =\nabla_1(\log(u_{xx}))-3, \\
K_2 &= \frac{u_{xy} u_{xxx}-u_{xx} u_{xxy}}{u_x u_{xx}^2}= \nabla_2(\log(u_{xx})),\\
K_3 %&= K_2 \left(1-2  \frac{u_{xy}}{u_{x}^2}\right)+2 \frac{u_{xx}  u_{yy}- u_{xy}^2}{u_{x}^4} +2 \frac{ u_{xy}  u_{xxy}- u_{xx}  u_{xyy}}{ u_{x}^3  u_{xx}} \\
&=K_2 \left(1-2  \frac{u_{xy}}{u_{x}^2}\right) -2 \frac{u_{xx}}{u_x^3}  \nabla_2(u_y)+\frac{2}{u_{x}^2} \nabla_2(u_{xy}),\\
K_4 &= \frac{1}{u_x^4} \left(u_{xx} \nabla_2(2u_{yy}-u_x u_y) -\nabla_2(u_{xy}/u_{xx}) u_{xx} (2u_{xy}-u_x^2)- \nabla_2(u_{xy}^2)     \right)
 \end{align*}
are independent differential invariants of the third order.

\com{Written in terms of $I_1,I_2,I_3$ the coefficients are
\begin{align*}
K_1 &= \frac{1}{\Delta} \big(I_1 I_2 ( \nabla_1(I_1)-2 I_1)+(I_2+I_3) ( \nabla_2(I_1)+ \nabla_1(I_3)- \nabla_1(I_1)+ \nabla_1(I_2)-I_1) \\&-I_1 ( \nabla_3(I_1)+ \nabla_2(I_3)+ \nabla_2(I_2)-I_1)+I_1^2 ( \nabla_1(I_2)-I_1)\big)+1\\
K_2 &=\frac{1}{\Delta}\big(I_2 (I_2+I_3)  \nabla_1(I_1)-I_1 I_2 ( \nabla_2(I_1)+ \nabla_1(I_3))\\&+(I_1-I_2-I_3) (2 I_1 I_2-I_1+ \nabla_3(I_1)+ \nabla_2(I_2)+ \nabla_2(I_3))\\&-I_1 (I_1-I_3) ( \nabla_1(I_2)-I_1)\big) -1\\
K_3 &=\frac{1}{\Delta} \big(2 I_2 (I_2+I_3) ( \nabla_2(I_1)- \nabla_1(I_1)+ \nabla_1(I_3))\\&+2 I_1 I_2 ( \nabla_1(I_1) I_2- \nabla_3(I_1)- \nabla_2(I_3)- \nabla_2(I_2)+ \nabla_1(I_2)-I_1 (I_1+2 I_2))\\&+2 (I_1-I_2-I_3) (I_1 I_2+ \nabla_1(I_2) I_3)\big)+ K_2+2I_1+2I_2
\end{align*}
\begin{align*}
K_4 &= \frac{1}{\Delta} \big((-I_1+I_2+I_3) ((1-2 I_2) ( \nabla_3(I_1)+ \nabla_2(I_3)+ \nabla_2(I_2))\\&-I_2 ( \nabla_2(I_1)+ \nabla_1(I_3))+(-I_1+I_3)  \nabla_1(I_2)+I_2^2  \nabla_1(I_1)\\&-I_1 (4 I_2^2-I_1-2 I_2+3 I_3)-I_1)+I_2^2 (I_2+I_3)  \nabla_1(I_1)\\&+I_1 I_2 ( \nabla_3(I_1)+ \nabla_2(I_3)+ \nabla_2(I_2)-2 I_2 ( \nabla_2(I_1)+ \nabla_1(I_3))\\&+(2 (-I_1+I_3))  \nabla_1(I_2)-I_1 (2 I_3-2 I_2+1))\big) +1+2 \nabla_2(I_2)+2I_3-I_2\\
\end{align*}
where
\[\Delta = I_1 I_2 (I_1-1) +I_2(I_2+I_3) + I_3(-I_1+I_2+I_3). \]}

The nine third-order differential invariants $\nabla_j(I_i)$ are independent, and together with $I_1,I_2,I_3$ they generate all differential invariants
of order three. In particular, $K_1,\dots,K_4$ can be expressed through them.

% In the next section it will become clear that every rational scalar differential invariants of any order is a rational function
% only of $I_2$ and its (invariant) derivatives.

Moreover, $I_1,I_2,I_3$ and $\nabla_1,\nabla_2,\nabla_3$ generate all rational scalar differential invariants of the $\mathcal{G}$-action on $\E$.

%%%%%%%%%%%%%%%%%%%%%%
\subsection{EW structure written in invariant coframe}\label{S3.3}

The invariant derivations $\nabla_1,\nabla_2,\nabla_3$ constitute a horizontal frame on an open subset in $\MS_2$. Let $\alpha^1,\alpha^2,\alpha^3$ be the dual horizontal coframe. The 1-forms $\alpha^i$ are defined at all points where $u_{xx}\neq 0$. Since $\alpha^1 \wedge \alpha^2 \wedge \alpha^3= -u_x^3 dt \wedge dx \wedge dy$, they determine a horizontal coframe outside the singular set $\Sigma_2=\{u_x=0,u_{xx}=0\} \subset \MS_2$.

In $\MS_2 \setminus \Sigma_2$ we can rewrite $g$ and $\omega$ in terms of the coframe $\alpha_1,\alpha_2,\alpha_3$. Then $g= g_{ij} \alpha^i \alpha^j$ and $\omega=\omega_i \alpha^i$, where $g_{ij}=g(\nabla_i,\nabla_j)$ and $\omega_i=\omega(\nabla_i)$. After rescaling the metric by a factor of $u_x^2$, we get the following expression.
%\begin{align*}
% g &= \frac{1}{u_x^2} \left(4 \alpha^1 \alpha^3 - \alpha^2 \alpha^2 +2\alpha^2 \alpha^3 +(4I_2-1) \alpha^3 \alpha^3 \right),\\
% \omega &= \alpha^2- \left(1-2\frac{u_x u_{xy}+u_y u_{xx}+2 v_x u_{xx}}{u_x^3}\right) \alpha^3.
%\end{align*}
%We can simplify this even more by doing a conformal rescaling taking  $g$ to $g'=u_x^2 g$. Then $\omega$ transforms to \[\omega' = \omega + 2 \hat d(\ln u_x) = \omega+\alpha^1+2 \nabla_3(\ln u_x) \alpha^3 ,\]
%and we get
\begin{align*}
g'&=  4 \alpha^1 \alpha^3 - \alpha^2 \alpha^2 +2\alpha^2 \alpha^3 +(4I_2-1) \alpha^3 \alpha^3, \\ \omega'&=2 \alpha^1+\alpha^2+(4I_2-1)\alpha^3.
\end{align*}

Thus, given any \EW structure whose 2-jet is in the complement of $\Sigma_2$ we may rewrite it in the form $(g',\omega')$, and we see that this expression only depends on $\alpha^1,\alpha^2,\alpha^3$ and $I_2$.  A consequence of these computations is the following theorem.
 \begin{theorem}
The field of rational $\mathfrak g$-differential invariants on $\MS$ is generated by the differential invariant $I_2$ together with the invariant derivations $\nabla_1, \nabla_2, \nabla_3$.
 \end{theorem}

The reason that we are able to generate the rest of the second-order differential invariants from these is that some algebraic combinations of the higher-order invariants  will be of lower order. In particular, we have the following identities relating $I_1,I_3$ to the invariants $K_i$ from the commutation relations of the invariant derivations.
 \begin{align*}
I_1 &= \nabla_1(I_2)+\frac{K_2+K_3}{2}-I_2 K_1, \\
I_3 &= (\nabla_1-\nabla_2)(I_2)+\frac{K_2+3K_3+2K_4}{4}+I_2 (K_2-K_1-1).
 \end{align*}

%\subsection{Horizontal coframe from invariants}
%From our three differential invariants of order two, we get the invariant horizontal coframe $\hat dI_1, \hat dI_2, \hat dI_3$, with dual frame denoted by $\hat \partial_{I_1},\hat \partial_{I_2},\hat \partial_{I_3}$. The derivations $\hat \partial_{I_i}$ are invariant derivations, called Tresse-derivatives. Define $\Xi_3$ as the set of points $\theta \in \MS_3$ such that $(\hat dI_1 \wedge \hat dI_2 \wedge \hat dI_3) |_\theta$ vanishes or is undefined. On $\MS_3 \setminus \Xi_3$, the three 1-forms defines a horizontal coframe. If we rewrite $g$ and $\omega$ in terms of this coframe, we get
%\[ g= \sum G_{ij} \hat dI_i \hat dI_j,\qquad  \omega = \Omega_1 \hat dI_1+\Omega_2 \hat dI_2+\Omega_3 \hat dI_3. \]
%Again, $g$ represents a conformal structure, and if we change representative to $g''=\frac{1}{G_{33}} g$, we get
%\begin{align*}
%g''&= \sum \frac{G_{ij}}{G_{33}} \hat dI_i \hat dI_j,\\
%\omega'' &= \Omega_1 \hat dI_1+\Omega_2 \hat dI_2+\Omega_3 \hat dI_3-\hat d(\ln G_{33}) \\
%&=\left(\Omega_1-\frac{\hat \partial_{I_1} (G_{33})}{G_{33}} \right) dI_1+\left(\Omega_2-\frac{\hat \partial_{I_2} (G_{33})}{G_{33}} \right) dI_2+\left(\Omega_3-\frac{\hat \partial_{I_3} (G_{33})}{G_{33}} \right) dI_3.
%\end{align*}
%All coefficients appearing in $g''$ and $\omega''$ are differential invariants of order three.

%%%%%%%%%%%%%%%%%%%%%%
\subsection{The equivalence-problem of \EW structures}\label{S3.4}
By Theorem \ref{algebraic} from the appendix and the global Lie-Tresse theorem \cite{KL2} the field of differential invariants separates generic orbits on $\tilde \MS = \MS_\infty \setminus \pi_{\infty,\ell}^{-1}(S)$ for some Zariski closed invariant subset $S \subset \MS_\ell$. Therefore, the description of the field of differential invariants is sufficient for describing the quotient equation $\tilde \MS /\mathcal{G}$.

In order to finish a description of the field of differential invariants one must find the (differential) syzygies in the differential field of scalar
invariants. Since all invariants are rational this can be done by brute force. Using $\nabla_1,\nabla_2,\nabla_3,I_1,I_2,I_3$ as the generating set
of the field of invariants, a simple computation with the \textsf{DifferentialGeometry} package of \textsc{Maple} shows that the twelve invariants
$I_k,\nabla_i(I_j)$ are functionally independent, so there are no syzygies on this level. There are five polynomial relations between
$I_i, \nabla_j(I_i), \nabla_k \nabla_j(I_i)$. Due to their length the expressions are not reproduced here, but they can be found in
the Maple file ancillary to the arXiv version of this paper.

There is another way to describe the quotient equation in our case, using the same approach as \cite{LY} and \cite{KS}. Take three independent differential
invariants $J_1,J_2,J_3$ of order $k$ (for instance $I_1,I_2,I_3$). Their horizontal differentials $\hat dJ_1,\hat dJ_2,\hat dJ_3$ determine a horizontal coframe on $\E_\ell\setminus S$ for some Zariski closed subset $S \subset \E_\ell$, $\ell >k$.
It is then possible, in the same way as in Section \ref{S3.3}, to rewrite the \EW structure in terms of this coframe:
\begin{align*}
g'&=  \sum G_{ij} \hat dJ_i \hat dJ_j, \qquad  \omega'= \sum \Omega_{i} \hat dJ_i .
\end{align*}
For one of the nonzero coefficients $G_{ij}$ we may,  after rescaling the metric, assume that $G_{ij}=1$. The quotient equation
$(\MS_\infty \setminus \pi_{\infty,\ell}^{-1}(S))/\mathcal{G}$ is obtained by adding to the Einstein-Weyl equation on $\mathbb R^3(x_1,x_2,x_3)$
the equations $\{J_i=x_i\}_{i=1}^3$.

For practical purposes the following approach solves the local equivalence problem for \EW structures of the form (\ref{shape}),
using the idea of a signature manifold \cite{COSTH}. Let $I_1,I_2,I_3$ be the basic invariants and $I_{ij} =\nabla_j(I_i)$ their derivations.
For a section $s\in \Gamma(\pi)$ let $\mathcal S_s\subset \mathbb R^{12}(z)$ be the image of the map
 $$M\ni x\mapsto \bigl(z_1=I_1(j_2(s))(x),\dots,z_4=I_{11}(j_3(s))(x),\dots,z_{12}=I_{33}(j_3(s))(x)\bigr).$$
For generic $s$ the manifold $\mathcal S_s$ is three-dimensional; it is called the signature of $s$.
If, in addition, the \EW structure $s$ is given by algebraic functions, then $\mathcal S_s$ is an algebraic manifold and
it can be defined by polynomial equations.

Let us call a section $s$ $I$-regular if $\hat{d}I_i|_s$ are defined and $(\hat{d}I_1\wedge\hat{d}I_2\wedge\hat{d}I_3)|_s\neq0$.
The invariant derivations $\nabla_j$ can be reconstructed from the twelve invariants $I_k,I_{ij}$, which in turn determine
all other differential invariants. Therefore two $I$-regular sections $s_1,s_2$ of $\pi$ are equivalent if and only if their signatures coincide.
In the algebraic case this is equivalent to equality of the corresponding polynomial ideals, and so this can be decided algorithmically.

%%%%%%%%%%%%%%%%
\section{Some particular \EW spaces}\label{S4}

Symmetries can be used to find invariant solutions of differential equations. They can be also used to
obtain explicit non-symmetric solutions: use a differential constraint consisting of several differential invariants
and solve the arising overdetermined system.
 % This is similar to algebraically special solutions in general relativity.
 % Then the symmetry group action can be used to obtain a bigger class of exact solutions.
In this setup the solutions come in a family, invariant under the symmetry group action, so in examples below
we normalize them using $\mathcal{G}$ to simplify the expressions.
Since use of symmetry gives a differently looking solution, but an equivalent \EW space, the generality does not suffer.

\medskip

{\bf 1.} We begin with the only relative invariant of order 1: $u_x=0$. This coupled with equation $F_1=0$ gives
$u_{yy}=0$, so $u=a(t)y+b(t)$. This can be transformed to $u=0$ by our pseudogroup $\mathcal{G}$. Then the second equation $F_2=0$
becomes the dispersionless Kadomtsev-Petviashvili (dKP), also known as the Khokhlov-Zabolotskaya equation in 1+2 dimensions \cite{KZ,KP}:
 $$
v_{tx}+v_x^2+vv_{xx}-v_{yy}=0.
 $$
This equation is integrable and has been extensively studied, see e.g. \cite{MS,DMT}.

Note that the orbit in $\E_2$ of lowest dimension, given by $\{u_x=0, u_{tx}=0, u_{xx}=0, u_{xy}=0, v_{xx}=0, v_{xy}=0\}$, leads to the solution
 $$
u= f_1(t) +f_2(t)y,\
v= f_3(t)+f_4(t) x+f_5(t)y+\tfrac12(f_4(t)^2 +2f_2(t)f_4(t)+\dot f_4(t))y^2
 $$
which is $\mathcal{G}$-equivalent to $(u,v)\equiv(0,0)$.

\medskip

{\bf 2.} Consider the special value of the first invariant $I_1=0$. The arising system $u_{xy}+v_{xx}=0$ has a solution
$u=w_x, v=-w_y$. Substitution of this into the modified Manakov-Santini system reduces it to the prolongation
of the first equation from the universal hierarchy of Mart\'inez Alonso and Shabat  \cite{MASh}:
 \begin{equation}\label{MASh}
w_{tx}+w_xw_{xy}-w_yw_{xx}-w_{yy}=0.
 \end{equation}
In fact, the equations $F_1=0$ and $F_2=0$ are $x$- and $y$-derivatives of the left-hand side $F$ of \eqref{MASh},
so we get the PDE $F=f(t)$ and the function $f(t)$ can be eliminated by a point transformation.

Equation \eqref{MASh} possesses a Lax pair and so is integrable by the inverse scattering transform.
Its hierarchy carries an involutive $GL(2)$-structure \cite{FKGL2}, and so is also integrable by twistor methods.
The method of hydrodynamic reductions \cite{FKh} can be exploited to obtain solutions $w$ involving arbitrary functions of one argument.

\medskip

{\bf 3.} Consider a stronger ansatz for the modified Manakov-Santini equation: $I_1=0$, $I_2=0$, $I_3=0$,
in addition to $F_1=F_2=0$. This overdetermined system can be analyzed by the \textsf{rifsimp} package of \textsc{Maple}.
The main branch is equivalent to the constraint $u_{xx}=0,u_{xy}=0,v_{xx}=0$. This can be explicitly solved.

Modulo the pseudogroup $\mathcal{G}$ the general solution to this system is
 $$
u=x+e^y,\qquad v=f(t)+h(t)e^{-y}.
 $$
Degenerations include the solution
 $$
u=0,\qquad v=\tfrac1{12}y^4+xy+h(t)
 $$
which is a partial solution to the dKP.

\medskip

{\bf 4.} Finally, consider an ansatz obtained by the requirement that all structure coefficients $K_1,\dots,K_4$
of the frame $\nabla_i$ on $\E_\infty$ and the coefficient $I_2$, arising in the expression of $(g,\omega)$,
are constants.

By the last formulae in \S\ref{S3.3} this case corresponds to constancy of all differential invariants
obtained from $I_1,I_2,I_3$ by $\nabla_i$-derivations. Also note that in this case $(\nabla_1,\nabla_2,\nabla_3)$
form a 3-dimensional Lie algebra $\mathfrak{s}$.

The obtained system $F_1=0, F_2=0, K_1=k_1, K_2=k_2, K_3=k_3, K_4=k_4, I_2=c$ is inconsistent for generic parameters
in the right-hand sides. Using the differential syzygies between the invariants and derivations, we further constrain those values.
The obtained system can be solved in \textsc{Maple}.

Let us restrict to the case, when the corresponding algebra $\mathfrak{s}=\mathfrak{sl}(2,\R)$
(otherwise $\mathfrak{s}$ is solvable).
This corresponds to very particular values of the parameters: $I_1=-\frac3{25}$, $I_2=\frac{21}{100}$, $I_3=-\frac{147}{500}$,
$K_1=1$, $K_2=0$, $K_3=\frac9{50}$, $K_4=-\frac9{500}$.

Modulo the pseudogroup $\mathcal{G}$ the general solution to this system is
 $$
u=y^{2/3} -\tfrac{10}3 xy^{-1},\quad
v=\tfrac25 xy^{-1/3} -\tfrac73 x^2y^{-2} +\tfrac{21}{25} y^{4/3} +(f(t)y^{1/3}+h(t))y^2.
 $$
A degeneration of this family gives the following family of solutions
 $$
u=-\tfrac{10}3 xy^{-1},\quad
v=-\tfrac73 x^2y^{-2} +(f(t)y^{1/3}+h(t))y^2.
 $$

% \smallskip

It shall be noted that we have essentially quotiented out the pseudogroup $\mathcal{G}$ (only the translation by $t$ remains in the latter cases)
because we integrated $\mathfrak{g}$ explicitly and have found a convenient cross-section of the action.

In the general case, when we impose an invariant differential constraint, the family of solutions can keep $\mathcal{G}$-invariance and the separation
of generic solutions can be done using the differential invariants obtained in \S\ref{S3.4}.

%%%%%%%%%%%%%%%%
\appendix
\section{Symmetry of algebraic PDEs}\label{algebraicity}
% \section{Symmetry of algebraic differential equations}\label{algebraicity}

% Here we prove the following important statement.
Let $\E\subset J^\infty\pi$ be a differential equation.
It is called algebraic if for every $a\in E=J^0\pi$ and every $k\in\N$
the fiber $\E^k_a\subset J^k_a\pi$ is an algebraic variety (maybe reducible).
Here we use the natural algebraic structure in the fibers $J^k\pi\to E$.

Note that the definition of algebraic pseudogroup in \cite{KL2} used an assumption that $\mathcal{G}$ acts transitively on $J^0\pi$.
For instance, this is the case if the bundle is trivial $\pi:E=\R^n(x)\times\R^m(u)\to\R^n(x)$ and the defining equations of $\E$
do not depend on $x,u$. It is also the case for the modified Manakov-Santini system (\ref{MS}).
We will not however rely on it in the proof below.

 \begin{theorem}\label{algebraic}
The symmetry pseudogroup $\mathcal{G}$ of an algebraic differential equation $\E$ is algebraic.
This means that the defining Lie equations of $\mathcal{G}$ are algebraic.
 \end{theorem}

In this formulation, by symmetries we mean either point or contact symmetries. The statement holds true also for mixed point-contact symmetries, as the ones appearing in the B\"acklund type theorem in \cite{AK}, and can be extended for generalized symmetries as those considered in \cite{AKO,KLV}.

 \begin{proof}
Without loss of generality we can assume $\E$ to be formally integrable, because addition of compatibility conditions does not change the symmetry.
The differential ideal $I_\E$ of the equation $\E$ is filtered by ideals $I_\E^i$ of functions on $J^i \pi$, and it is completely determined by
$I_\E^k$ for some $k$. By the assumption there exist generators $F_1,\dots,F_r$ of $I_\E^k$ that are algebraic in jet-variables $u_\sigma$, $|\sigma|>0$,
over any point $a=(x,u)\in E$, and from now on we restrict to a single point $a \in E$.

Let $\vp:E\to E$ be a local diffeomorphism (point transformation) with $\vp(a)=a$. It is a symmetry if $\vp^*F_i\in I_\E^k$ for every $i=1,\dots,r$, where we tacitly omitted the notation for prolongation.
For each $i$, the membership problem is algorithmically solvable by the Gr\"obner basis method, and the condition for membership is a set of algebraic relations. Unite those
by $i$. Decompose each relation by all jet-variables $u_\sigma$, $|\sigma|>0$ and collect the coefficients. This gives a finite number of algebraic differential equations
on the components of $\vp$. Their orders do not exceed the maximal order of $F_i$, because of the prolongations of $\vp$ involved.
This is the set of Lie equations defining $\mathcal{G}$, and the claim follows.

In the case of a contact diffeomorphism $\vp:J^1\pi\to J^1\pi$, when $u$ is one-dimensional, the decomposition has
to be done with respect to $u_\sigma$, $|\sigma|>1$. The rest of arguments is the same.
 \end{proof}

%%%%%%%%%%%%%%%%%%%%%%%%%%%%%%%%%%%%%%%%%%%%%%%%%%%%%%%%%%%%%%%%%%%%%%%%%%


\begin{thebibliography}{11}
\footnotesize

\bibitem{ALV}
D. Alekseevskij, V. Lychagin, A. Vinogradov, {\it Basic ideas and concepts of differential geometry\/}, Encycl. math. sciences {\bf 28}, Geometry~1, Springer (1991).

\bibitem{AKO}
I.\,M. Anderson, N. Kamran, P. Olver, {\it Internal, external and generalized symmetries\/},
Adv. Math. {\bf 100}, 53--100 (1993).

\bibitem{AK}
I.\,M. Anderson, B. Kruglikov, {\it Rank 2 distributions of Monge equations:
symmetries, equivalences, extensions\/}, Adv. Math. {\bf 228}, issue 3, 1435--1465 (2011).

\bibitem{COSTH}
E. Calabi, P. Olver, C. Shakiban, A. Tannenbaum, S. Haker, {\it Differential and Numerically Invariant Signature Curves Applied to Object Recognition}, International Journal of Computer Vision, {\bf 26}, issue 2, 107-135 (1998).

\bibitem{Calderbank}
D.M.J. Calderbank, {\it Integrable background geometries}, SIGMA {\bf 10}, 034 (2014).

\bibitem{Cartan}
E. Cartan, {\it Sur une classe d'espaces de Weyl},  Ann. Sci. \'Ecole Norm. Sup. (3) {\bf 60}, 1-16 (1943).

\bibitem{DFK}
M. Dunajski, E.V. Ferapontov, B. Kruglikov, {\it On the Einstein-Weyl and conformal self-duality equations}, Journ. Math. Phys. {\bf 56}, 083501 (2015).

\bibitem{DMT}
M. Dunajski, L.J. Mason, P. Tod, {\it Einstein-Weyl geometry, the dKP equation and twistor theory}, Journal of Geometry and Physics {\bf 37}, no. 1-2, 63-93 (2001).

\bibitem{FKh}
E.V.\ Ferapontov, K.R.\ Khusnutdinova, {\it On the integrability of (2+1)-dimensional quasilinear systems},
Comm. Math. Phys. \textbf{248}, 187-206 (2004).

\bibitem{FK}
E.V.\ Ferapontov, B.S.\ Kruglikov,  {\it Dispersionless integrable systems in 3D and Einstein-Weyl geometry}, Journal of Differential Geometry  {\bf 97}, 215-254  (2014).

\bibitem{FKGL2}
E.V.\ Ferapontov, B.S.\ Kruglikov,  {\it Dispersionless integrable hierarchies and GL(2,R) geometry},
arXiv:1607.01966 (2016).

\bibitem{H}
N.J.\ Hitchin, {\it Complex manifolds and Einstein’s equations},
Twistor Geometry and Non-linear Systems, Lecture Notes in Math. {\bf 970} Springer (1982).

\bibitem{JT}
P.E. Jones, K.P. Tod, {\it Minitwistor spaces and Einstein-Weyl spaces}, Class. Quantum Grav. {\bf 2}, no. 4, 565-577 (1985).

\bibitem{KP}
B.B.\ Kadomtsev, V.I.\ Petviashvili, {\it On the stability of solitary waves in weakly dispersive media\/},
Sov.\ Phys.\ Dokl.\ {\bf 15}, 539-541 (1970).

\bibitem{KZ}
R.V.\ Khokhlov, E.A.\ Zabolotskaya, {\it Quasi-plane waves in the nonlinear acoustics of confined beams\/},
Sov.\ Phys.\ Acoust.\ {\bf 15}, 35–40 (1969).

\bibitem{KLV}
I. Krasilshchik, V. Lychagin, A. Vinogradov, {\it Geometry of jet spaces and nonlinear partial differential equations}, Gordon and Breach (1986).

\bibitem{KL1}
B. Kruglikov, V. Lychagin, {\it Geometry of Differential equations\/}, Handbook of Global Analysis, Ed. D.Krupka, D.Saunders, Elsevier, 725-772 (2008).

\bibitem{KL2}
B. Kruglikov, V. Lychagin, {\it Global Lie-Tresse theorem}, Selecta Mathematica {\bf 22}, 1357-1411 (2016).

\bibitem{KS}
B. Kruglikov, E. Schneider, {\it Differential invariants of self-dual conformal structures}, Journal of Geometry and Physics, {\bf 113}, 176-187 (2017).

\bibitem{LY}
V. Lychagin, V. Yumaguzhin, {\it Invariants in Relativity Theory}, Lobachevskii Journal of Mathematics {\bf 36}, no.3, 298-312 (2015).

\bibitem{MASh}
L.\ Mart\'inez Alonso, A.B.\ Shabat, {\it Energy-dependent potentials revisited: a universal hierarchy of hydrodynamic type\/},
Phys.\ lett.\ A {\bf 300}, 58-54 (2002).

\bibitem{MS}
S.V. Manakov, P.M. Santini, {\it Cauchy Problem on the Plane for the Dispersionless Kadomtsev-Petviashvili Equation}, JETP Lett. {\bf 83} (2006) 462-6.

\bibitem{O}
P. Olver, {\it Equivalence, Invariants, and Symmetry}, Cambridge University Press, Cambridge (1995).

\bibitem{T}
T.Y. Thomas, {\it The Differential Invariants of Generalized Spaces}, Cambridge University Press, Cambridge (1934).

\bibitem{W}
H. Weyl, {\it Reine Infinitesimalgeometrie}, Math. Z. {\bf 2}, 384-–411 (1918).

\end{thebibliography}
\end{document}